\documentclass[a4paper, 12pt,oneside,reqno]{amsart}
\usepackage[a4paper]{geometry}
\usepackage{graphicx}
\usepackage[matrix,arrow,curve,cmtip]{xy}

\usepackage{txfonts}
\DeclareMathAlphabet{\mathcal}{OMS}{cmsy}{m}{n} 

\usepackage{
  amsmath,  
  amssymb,  
  amsthm,   
  tikz,     
  fullpage, 
  youngtab, 
  ytableau, 
  thmtools, 
  hyperref, 
  cite,     
  url       
}

\usepackage{amscd}
\usepackage{euscript}
\usepackage{latexsym}
\usepackage[cp1251]{inputenc}
\usepackage[english]{babel}
\usepackage{mathrsfs}
\usepackage{graphicx}
\usepackage{keyval}
\usepackage{mathtools}
\usepackage{tikz}
\usetikzlibrary{arrows,shapes,snakes,automata,backgrounds,petri,through,positioning}
\usetikzlibrary{intersections}
\usepackage[symbol*]{footmisc}
\usepackage{verbatim}
\usepackage{tikz-cd}

\DeclareMathOperator{\C}{C}

\DeclareMathOperator{\Hom}{Hom}

\def\End{\mathop {\fam 0 End} \nolimits}
\def\oo#1{\mathbin {{}_{(#1)}}}

\usepackage{etoolbox}
\patchcmd{\thebibliography}{\section*}{\paragraph}{}{}

\numberwithin{equation}{section}

\newtheorem{theorem}{Theorem}[section]
\newtheorem{lemma}[theorem]{Lemma}
\newtheorem{proposition}[theorem]{Proposition}
\newtheorem{corollary}[theorem]{Corollary}

\theoremstyle{definition}

\newtheorem{example}[theorem]{Example}
\newtheorem{remark}[theorem]{Remark}

\title{Hochschild cohomology of the algebra of conformal endomorphisms}
\author{H. Alhussein$^{1),2)}$}
\author{P. Kolesnikov$^{3),4)}$}

\address{$^{1)}$Siberian State University of Telecommunication and Informatics, Novosibirsk, Russia.}
\address{$^{2)}$Novosibirsk State University of Economics and Management, Novosibirsk,  Russia.}
\address{$^{3)}$Sobolev Institute of Mathematics, Novosibirsk, Russia.}
\address{$^{4)}$Higher School of Economics, Moscow, Russia.}

\def\Cend{\mathop {\fam 0 Cend} \nolimits}
\def\oo#1{\mathbin {{}_{(#1)}} }

\begin{document}

\maketitle

\section{Introduction}

The notion of a conformal (Lie) algebra emerged in \cite{KacValgBeginners} as a tool in the theory 
of vertex algebras which goes back to mathematical physics \cite{BPZ1983} and representation theory
(see, e.g., \cite{Borch}). From the algebraic point of view, the structure of a vertex algebra 
is a breed of two structures: a differential left-symmetric algebra and a Lie conformal algebra \cite{BK-Field2002}.

The structure theory of (finite) Lie conformal algebra was developed in \cite{DK1998}, irreducible representations of 
simple and indecomposable semisimple finite Lie conformal algebras were described in \cite{ChengKac}. 
Given a finite conformal module $M$ over a Lie conformal algebra $L$, the representation of $L$ on $M$ 
is a homomorphism from $L$ to the Lie conformal algebra of conformal endomorphisms $\mathrm{gc}\,(M)$, 
see \cite[Ch.~2]{KacValgBeginners}. The latter is an analogue of the ``ordinary'' Lie algebra $\mathrm{gl}\,(V)$ 
of a linear space $V$ in the category of conformal algebras. As in ordinary algebras, $\mathrm{gc}\,(M)$
is the commutator algebra of an {\em associative} conformal algebra $\Cend (M)$. Thus the study of associative conformal 
algebras (and $\Cend (M)$, in particular) is essential for representation theory of Lie conformal algebras 
and, as a corollary, for vertex algebras theory.
A systematic study of $\Cend (M)$ was performed in \cite{BKL2003}, its simple subalgebras were described in \cite{Kol2006Adv}. 
The most interesting case is when $M$ is a free $H$-module of rank $k$, 
then $\Cend (M)$ is denoted $\Cend_k$. 
This system plays the same role in the theory 
of conformal algebras as the matrix algebra $M_k(\Bbbk )$
does in the ordinary algebra.

The homological studies for conformal algebras starts from the paper \cite{BKV}. 
Conceptually, to define (co)chains, (co)cycles, and (co)boundaries for a particular class of algebras over a field $\Bbbk $,
one needs to know what a multilinear mapping is, 
how to combine such mappings, and how symmetric groups act on multilinear mappings.
All these notions have their analogues in the category of modules over cocommutative bialgebras, that is, 
these are pseudo-tensor categories \cite{BDK}.
In particular, the definition of Hochschild cohomologies for an associative algebra in the pseudo-tensor category 
over the polynomial bialgebra $H=\Bbbk [\partial ]$, where $\partial $ is a primitive element, coincides 
with the definition of Hochschild cohomology of associative conformal algebras in \cite{BKV}.

It is well-known since \cite{Hoch1943} that for the associative algebra $\End (V)$ of linear transformations of a finite-dimensional space $V$ all $n$th Hochschild cohomology groups are trivial for $n\ge 1$.
The problem of description of conformal Hochschild cohomologies of $\Cend (M)$ for a finite $H$-module $M$ 
was stated in \cite{BKV}. 
In \cite{Dolg2009}, it was shown that the second Hochschild cohomology group of $C=\Cend (M)$ is trivial for all 
conformal bimodules over~$C$, which was a partial solution of the problem from \cite{BKV}.
The purpose of this paper is to complete solving this problem and prove 
that all $n$th Hochschild cohomology groups of $\Cend (M)$ for $n\ge 2$ with coefficients 
in all conformal bimodules over $\Cend (M)$.
Note that the classical argument (see \cite{Hoch1943}) based on the isomorphism 
$\mathrm H^n(A,M)\simeq \mathrm H^{n-1}(A, \mathrm{Hom}\,(A,M))$
does not work for conformal algebras since 
 the analogue of $\mathrm{Hom}$ denoted $\mathrm{Chom}$ (see \cite{KacValgBeginners}) does not carry a structure of conformal bimodule 
due to locality issues.

As shown in \cite{BKV}, the calculation of conformal Hochschild cohomology 
$\mathrm H^\bullet (C,M)$ of an associative conformal algebra $C$ with coefficients in a conformal bimodule $M$ over $C$ 
is based on the ordinary Hochschild cohomology 
$\mathrm H^\bullet (\mathcal A_+(C), M)$, where $\mathcal A_+(C)$ is the positive part of the coefficient algebra of~$C$. 

For $C=\Cend_k$, 
the positive part $\mathcal A_+(\Cend_k)$ of its coefficient algebra is isomorphic to the matrix algebra 
over the first Weyl algebra $W_1$, i.e., the unital associative algebra 
generated by two elements $p$, $q$ such that $qp-pq=1$.

The series of Weyl algebras (and, in particular, the first one) is under intensive study in various areas of mathematics. 
Homological properties of these algebras 
were considered, for example, in \cite{GHL, Rine, Hart}. For instance, the global dimension of the 
Weyl algebra $W_n$, $n\ge 1$, essentially depends on the characteristic of the base field.
One of the by-products of this paper is an
explicit computation of the 3rd Hochschild cohomology group of the first Weyl algebra by means of the Anick resolution. We apply the Morse matching method to transform a bar-resolution of the first Weyl algebra 
into its Anick resolution and calculate explicitly 
$\mathrm H^3(W_1, M)$ for an arbitrary $W_1$-bimodule~$M$.

As a result, we solve a problem stated in \cite{BKV}
on the computation of Hochschild cohomologies of 
the conformal algebra $\Cend_k$: we prove 
$\mathrm H^n(\Cend_k, M)=0$ for all $n\ge 2$ and for all 
conformal bimodules $M$ over $\Cend_k$.

\section{Morse matching method for constructing the Anick resolution}\label{sec:MorseMatching}

The idea of D. Anick on the construction of a relatively 
small free resolution for an augmented algebra 
has shown its effectiveness in a series of applications
\cite{AK2020,A2021,A2022-Cn, A2022, Akl, lopatkin}.
Let us briefly observe the main points of this construction
and its application to the computation of Hoch\-schild cohomologies of associative algebras.
Suppose $\Lambda $ 
is a unital associative algebra equipped with a homomorphism 
$\varepsilon: \Lambda \to \Bbbk $, 
$\varepsilon(1)=1$ (augmentation).
Denote by $A$ the cokernel $\Lambda/\Bbbk $ of the inverse embedding $\eta : \Bbbk \to \Lambda $
and consider the two-sided bar resolution of free $\Lambda$-bimodules
\[
0\leftarrow \Bbbk \leftarrow \mathrm{B}_{0}
\leftarrow  \mathrm{B}_1 \leftarrow \dots
\leftarrow \mathrm{B}_n \leftarrow \mathrm{B}_{n+1} \leftarrow \dots ,
\]
where $\mathrm{B}_0 = \Lambda\otimes \Lambda  $, 
$\mathrm{B}_n 
= \Lambda \otimes A^{\otimes n}\otimes \Lambda $ for 
$n\ge 1$. 
We will denote a tensor $a_1\otimes \dots \otimes a_n \in A^{\otimes n}$ as $[a_1|\ldots |a_n]$ and omit the tensor 
product signs between $\Lambda $ and $A^{\otimes n}$.
The arrows $d_{n+1}: \mathrm{B}_{n+1} \to \mathrm{B}_n$
are $\Lambda $-bimodule homomorphisms given by 
\begin{equation}\label{eq:Bar-Differential}
d_{n+1}[a_1|\ldots |a_{n+1}]
= a_1[a_2|\ldots |a_{n+1}] 
+\sum\limits_{i=1}^n (-1)^i[a_1| \ldots |a_ia_{i+1}|\ldots |a_{n+1}] 
+ (-1)^{n+1} [a_1|\ldots |a_n] a_{n+1},
\end{equation}
for $n>0$, and 
\[
d_1: [a]\mapsto a\otimes 1 - 1\otimes a, \quad 
d_0: a\otimes b \mapsto \varepsilon(ab).
\]
If $M$ is an arbitrary unital $\Lambda $-bimodule 
then 
\[
\Hom_{\Lambda{-}\Lambda} (\mathrm B_n, M) \simeq 
\Hom (A^{\otimes n}, M)
\]
as linear spaces, and for every 
$\varphi \in \Hom_{\Lambda{-}\Lambda} (\mathrm B_n, M)$
the composition 
$\varphi d_{n+1}: \mathrm B_{n+1} \to M$
corresponds to the $\Bbbk $-linear map 
$\Delta^n (\varphi ): A^{\otimes (n+1)}\to M$
which is given by the Hochschild differential formula.

Therefore, if we start with an associative algebra $A$,
join an exterior identity to get 
$\Lambda = A\otimes \Bbbk 1$ with $\varepsilon(A)=0$, 
then the cochain complex 
\[
\big (
\Hom_{\Lambda{-}\Lambda} (\mathrm B_\bullet , M), \Delta^\bullet 
\big )
\]
coincides with Hochschild complex 
$\mathrm {C}^\bullet (A,M)$.

The bar resolution $(\mathrm B_\bullet , d_\bullet)$
is easy to construct but it is too large 
for particular computations. Therefore, it is reasonable 
to replace $(\mathrm B_\bullet , d_\bullet)$
with a smaller but homotopy equivalent resolution, 
e.g., the {\em Anick resolution} $(\mathrm A_\bullet, \delta_\bullet)$, 
\[
0\leftarrow \Bbbk \leftarrow \mathrm A_0 \leftarrow \mathrm A_1 \leftarrow \dots \leftarrow \mathrm A_n \leftarrow \mathrm A_{n+1} \leftarrow \dots, 
\quad \delta_{n+1}: \mathrm A_{n+1}\to \mathrm A_n.
\]
Then, given an $A$-bimodule (hence, a unital $\Lambda $-bimodule), the cohomologies of 
the complex 
\[
\big ( 
\Hom_{\Lambda{-}\Lambda} (\mathrm A_\bullet, M), \Delta^\bullet
\big ), \quad \Delta^{n}\varphi = \varphi \delta_{n+1}, 
\ \varphi \in \Hom_{\Lambda{-}\Lambda} (\mathrm A_n, M),
\]
coincide with the Hochschild cohomologies 
$\mathrm H^\bullet (A,M)$.

Suppose $X$ is a set of generators of the algebra $A$.
Denote by $X^*$ the set of nonempty words in $X$, 
and let $\Bbbk \langle X\rangle $ stand for 
the linear span of $X^*$, this is the free associative 
algebra generated by~$X$.

Let $\Sigma \subset \Bbbk \langle X\rangle $
be a Gr\"obner--Shirshov basis of $A$ relative 
to an appropriate monomial order (e.g., deg-lex order).
We will denote by $V = \overline \Sigma $ the set of 
principal parts of relations from~$\Sigma $
(called {\em obstructions}).
Recall that 
$\mathrm A_0 = \mathrm B_0 = \Lambda \otimes \Lambda $, 
$\mathrm A_n = 
\Lambda \otimes \Bbbk V^{(n-1)}\otimes \Lambda $, 
where $V^{(k)}$ stands for the set of {\em Anick $k$-chains}.
By definition (see \cite{Anick1983}), 
$V^{(0)}=\{[x] \mid x\in X\}$, 
$V^{(1)} = \{[x|s] \mid x\in X, s\in X^*, xs\in V\}$, 
and for $k\ge 2$ the set $V^{(k)}$
is constructed on the words in $X^*$ obtained by 
consecutive ``hooking'' of the words from 
$\overline \Sigma $.

This definition becomes transparent in the case when the defining relations $\Sigma $ contain at most quadratic monomials, so that all words in 
$V $ are of length two.
For $n \ge 1$, an Anick $n$-chain is a word 
$v=[x_{0}|\ldots |x_{n}] \in X^{*}$ such that $x_ix_{i+1}\in V$ for $i=0,\ldots,n-1$.

\begin{example}\label{exmp:UnivEnvelope}
Let $\mathfrak g$ be a Lie algebra over $\Bbbk $ with a linearly ordered basis $X$. Denote $[x,y]\in \Bbbk X$,
$x,y\in X$,
the Lie product in $\mathfrak g$. 
Set $\Sigma = \{ xy-yx-[x,y] \mid x,y\in X, x>y \}$,
$A=\Bbbk \langle X\rangle /(\Sigma )$. Then 
$\Lambda = A\oplus \Bbbk 1 $ is exactly the universal 
enveloping associative algebra $U(\mathfrak g)$.
Then $V^{(k)} = \{[x_0|x_1|\ldots |x_k] \mid x_0>x_1>\dots >x_k, x_i\in X \}$.
The elements of $V^{(k)}$ are in obvious one-to-one correspondence with the basis of $\wedge^{k+1}\mathfrak g$.
\end{example}

The Anick differentials $\delta_{n+1}:\mathrm A_{n+1}\to \mathrm A_n$ were computed in \cite{Anick1983} by means of a complicated 
inductive procedure. In order to make this computation easier, in \cite{JollWelker}  and, independently, in \cite{Skoldb}, it was 
proposed to use algebraic discrete Morse theory developed in \cite{formancell, formanguide}
to construct a smaller complex (of free modules) which is homotopy equivalent to a given one. 
In particular, given a bar resolution of an augmented algebra $\Lambda $, 
the resulting complex is the Anick resolution.

The Morse matching method for computing the Anick resolution 
\cite{JollWelker}, \cite{Skoldb} is also described in 
\cite{lopatkin, A2022}. 
In a few words, the problem is to choose a set of edges in 
the weighted directed graph describing the structure of 
the bar resolution. Then one has to transform the graph 
by means of inverting the matched edges. 
Inverting means not only switch of direction,
but also replacing the weight $c$ of the matched edge 
with $-c{-1}$.
In the resulting graph, the non-matched vertices 
(critical cells) are exactly the Anick chains. 
Finally, in order to calcuate the Anick differential $\delta_{n+1}$
on a chain $w$ from $V^{(n)}$ one has to track all paths 
from $w$ to vertices from $V^{(n-1)}$.
The weight of each path is equal to the product of the weights 
of all its edges.

\begin{example}\label{exmp:Heisen-3}
Let $\mathfrak g = H_3$ be the Heisenberg Lie algebra.
The universal enveloping algebra  $U(H_3)$ 
is generated by the elements $x,y,z$, 
relative to the following relations:
\[
xy=yx+z,\quad xz=zx,\quad yz=zy.
\]
Assume $x>y>z$. Then the Anick $n$-chains are:
\[
V^{(1)}=\{[x|y],[x|z],[y|z]\},
\quad
V^{(2)}=\{[x|y|z]\},
\quad
V^{(n)}=\emptyset,\ n\ge 3.
\]
In order to compute $\delta_3[x|y|z]$, consider a fragment of the 
bar resolution graph and choose a Morse matching (dashed edges on Fig.~\ref{Fig1}).
Tracking the paths and collecting similar terms lead to 
the following answer:
\[ 
\delta_3[x|y|z]=x[y|z]-[y|z]x+[x|z]y -y[x|z]+z[x|y]-[x|y]z.
\]
\begin{figure}
\includegraphics{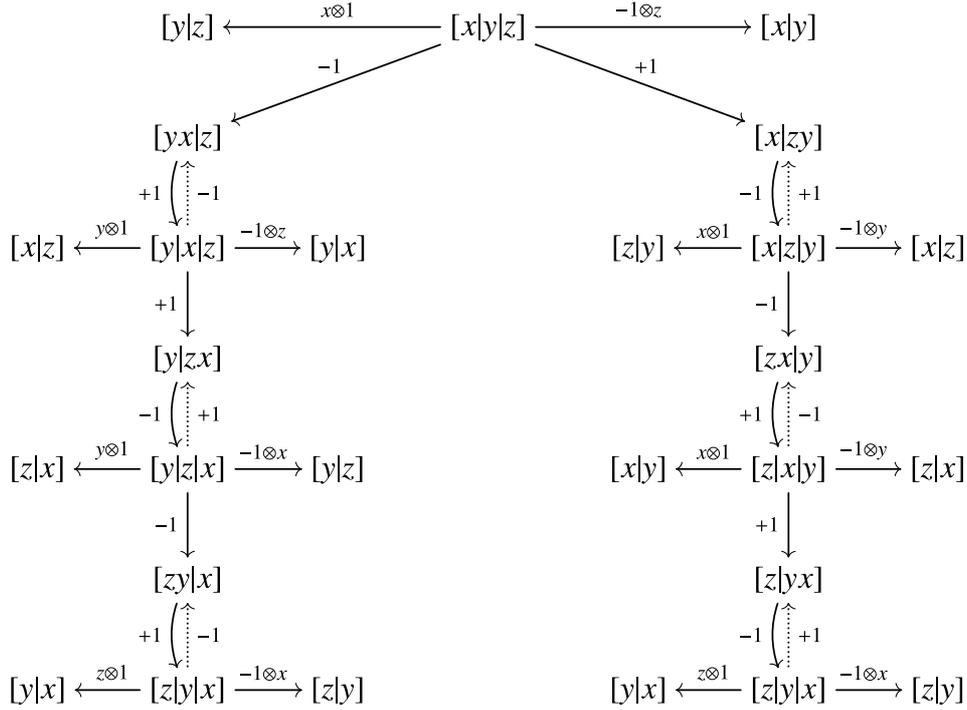}
\caption{Calculating the Anick differential of $[x|y|z]$ for $U(H_3)$}\label{Fig1}
\end{figure}
\end{example}

\begin{remark}
The differential in Example~\ref{exmp:Heisen-3}
corresponds to ``two-sided'' resolution. The restriction 
to the left module case (i.e., when multiplication by $x,y,z$ from the right is zero) leads us exactly to the Chevalley--Eilenberg differential for the Lie algebra $H_3$. This is a general observation: given a Lie algebra $\mathfrak g$,
the ``left'' Anick resolution for $U(\mathfrak g)$ 
coincides with the Chevalley--Eilenberg resolution for~$\mathfrak g$.
\end{remark}

When applied to the settings of Example~\ref{exmp:UnivEnvelope}, 
the Anick differential for $U(\mathfrak g)$ coincides with the Chevalley--Eilenberg differential for the Lie algebra~$\mathfrak g$.

\section{Conformal endomorphisms and the 1st Weyl algebra}

From now on, $\Bbbk $ is a field of characteristic zero, 
$H=\Bbbk [\partial ]$ is the polynomial algebra in one variable. 

Suppose $V$ and $M$ are two $H$-modules.
A {\em conformal homomorphism} \cite{KacValgBeginners} from $V$ to $M$
is a $\Bbbk $-linear map 
\[
\varphi _\lambda : V\to M[\lambda ]=\Bbbk [\partial,\lambda ]\otimes _H M
\]
such that 
\[
\varphi_\lambda (f(\partial ) v) = f(\partial+\lambda )\varphi_\lambda (v)
\]
for all $v\in V$, $f=f(\partial) \in H$.

If $M=V$ then the space of all conformal homomorphisms from $V$ to $M$
is denoted $\Cend (V)$. This is also an $H$-module:
\[
(\partial \varphi)_\lambda  = -\lambda \varphi_\lambda ,
\]
and if $V$ is a finitely generated $H$-module 
then $\Cend (V)$ is an {\em associative conformal algebra} \cite{KacValgBeginners}: 
for every $\varphi,\psi \in \Cend(V)$ we have 
\[
(\varphi \oo\lambda \psi ) \in \Cend (V)
\]
defined by the rule 
\[
(\varphi \oo\lambda \psi )_\mu  = \varphi_\lambda \psi _{\mu+\lambda }. 
\]
If $V$ is a free $H$-module of rank $k\in \mathbb N$ then 
$\Cend(V)$ is denoted $\Cend_k$.

Up to an isomorphism (see \cite{BKL2003, Kol2006Adv}), one may identify $\Cend_k$ with the space of all $(k\times k)$-matrices over the polynomial ring $\Bbbk [\partial , x]$ equipped with the operation 
\[
f(\partial, x)\oo\lambda g(\partial , x) = 
f(-\lambda , x)g(\partial+\lambda , x+\lambda ), 
\]
$f,g\in \Bbbk [\partial, x]$. For matrices, the operation $(\cdot\oo\lambda \cdot)$ is extended by the ordinary row-column rule.

Let $H$ act from the right on the Lawrent polynomials $\Bbbk [t,t^{-1}]$ in such a way that $\partial = -d/dt$. 
For every conformal algebra $C$ in the sense of \cite{KacValgBeginners}, 
one may define the {\em coefficient algebra} $\mathcal A(C)$
as the linear space 
$\Bbbk[t,t^{-1}]\otimes _H C$
equipped with the multiplication 
\begin{equation}\label{eq:CoeffProd}
a(n)b(m) = \sum\limits_{s\ge 0} \binom{n}{s} (a\oo{s} b)(n+m-s)
\end{equation}
where 
$t^n\otimes _H a = a(n)$ for $a\in C$, $n\in \mathbb Z$, 
and $(a\oo s b) $ stands for the coefficient at $\lambda^s/s!$ of 
$(a\oo\lambda b)$, $a,b\in C$.
For polynomials from $\Cend_1$, for example, we have 
\[
f(x)\oo{s} g(x) = f(x)\dfrac{d^s}{dx^s} g(x)
\]
by the Taylor formula.

The subspace of $\mathcal A(C)$ spanned by all $a(n)$, $n\ge 0$, $a\in C$, is a subalgebra of $\mathcal A(C)$ denoted $\mathcal A_+(C)$.
For instance, $\mathcal A(\Cend_1) = \Bbbk [t,t^{-1},x]$ as a linear space, the isomorphism identifies 
$t^n\otimes _H x^m$, $n\in \mathbb Z$, $m\in \mathbb Z_+$,
with $x^m t^n \in \Bbbk [t,t^{-1},x]$.
The product of two such monomials is calculated via 
\eqref{eq:CoeffProd}. For example, 
\[
t^n \cdot xt^m
 = (1\oo{0} x)t^{n+m} + n (1\oo{1} x) t^{n+m-1}
 = x t^{n+m} + n t^{n+m-1}, 
\]
so $tx = xt +1$, 
$t^{-1}x = xt^{-1} - t^{-2}$, etc.
Hence, $\mathcal A(\Cend_1)$
is isomorphic to 
the localization
of the first Weyl algebra 
$W_1 = \Bbbk \langle p,q\mid qp-pq=1\rangle $
relative to the multiplicative set
$\{q^s\mid s\ge 0\}$.
The positive part $\mathcal A_+(\Cend_1)$
is isomorphic to the Weyl algebra itself, so 
$\mathcal A_+(\Cend_k) \simeq M_k(W_1)$.

Let $C$ be an associative conformal algebra, 
and let $M$ be a conformal bimodule over~$C$.
Then $M$ is a bimodule over the ordinary associative algebra 
$A=\mathcal A_+(C)$, the action is given by
\[
a(n)\cdot u = a\oo{n} u,\quad u \cdot a(n) = \{u\oo{n} a\} = \sum\limits_{s\ge 0} (-1)^{n+s} \dfrac{1}{s!} \partial^s (u\oo{n+s} a),
\]
for 
$u\in M$, $a\in C$, $n\in \mathbb Z_+$.

The {\em basic Hochschild complex} \cite{BKV} 
of $C$ with coefficients in~$M$
is isomorphic to the Hochschild complex 
of $A=\mathcal A_+(C)$ with coefficients in the same bimodule~$M$.
There is a linear map 
\[
D_n : \C^n(  A,M) \to \C^n( A ,M)
\]
given by
\[
(D_n f)(a_1(m_1),\ldots, a_n(m_n) )
= \partial f(a_1(n_1),\ldots, a_n(m_n))
+ \sum\limits_{i=1}^n 
m_i f(a_1(n_1),\ldots, a_i(m_i-1), \ldots, a_n(m_n)),
\]
for $f\in \C^n( A ,M)$.
The maps $D_n$ are induced by the derivation $\partial: a(m)\mapsto -ma(m-1)$ on the algebra $A$.
Since $D_{n+1}\Delta^n = \Delta^n D_n$, the image 
$D_\bullet \C^\bullet (A,M)$ is a subcomplex 
of $\C^\bullet (A,M)$, and the quotient 
\begin{equation}\label{eq:RestrictedComplex}
\overline{\C}^\bullet(A,M) = \C^\bullet (A,M) / D_\bullet \C^\bullet (A,M)
\end{equation}
is isomorphic to the {\em reduced Hochschild complex}
of the conformal algebra $C$ 
(see \cite[Theorem 6.1, Corollary 6.1]{BKV}).

\begin{proposition}\label{prop:MainTool}
If $C$ is an associative conformal algebra, 
$A = \mathcal A_+(C)$,
$M$ is a conformal bimodule over $C$, and
$\mathrm H^q(A,M)=0$ for all $q\ge 3$,
then 
$\mathrm H^q(\overline{\C}^\bullet (A,M)) = 0$
for all $q\ge 3$.
\end{proposition}

\begin{proof}
The short exact sequence 
\[
0\to 
D_\bullet \C^\bullet (A,M) \to 
 \C^\bullet (A,M) \to
\overline{\C}^\bullet(A,M) 
\to 0
\]
gives rise to the long exact sequence of cohomologies
\[
\begin{aligned}
\dots \to{}& \mathrm H^q (D_\bullet \C^\bullet (A,M))
\to \mathrm H^q (\C^\bullet (A,M))
\to \mathrm H^q (\overline{\C}^\bullet (A,M)) \\
\to{}& \mathrm H^{q+1} (D_\bullet \C^\bullet (A,M))
\to \mathrm H^{q+1} (\C^\bullet (A,M))
\to \mathrm H^{q+1} (\overline{\C}^\bullet (A,M)) 
\to \dots 
\end{aligned}
\]
By \cite[Proposition 2.1]{BKV}, the complexes $\C^\bullet =\C^\bullet (A,M)$ and $D_\bullet \C^\bullet $ are isomorphic in positive degrees. 
Hence, under the conditions of the statement, 
$\mathrm H^q (\overline{\C}^\bullet (A,M))$, $q\ge 3$, is clamped 
between zeros, thus it is zero itself.
\end{proof}

\section{Two-sided Anick resolution for the first Weyl algebra}

In this section, we apply the Morse matching method described in 
Section \ref{sec:MorseMatching} to compute the 3rd Hochschild cohomology of the first Weyl algebra with coefficients in an arbitrary bimodule.

The Weyl algebra $W_1$ 
is generated by the elements $q,p,e$, 
relative to the following relations:
\[
qp=pq+e,\quad pe=p,\quad qe=q,\quad eq=q,\quad ep=p,\quad ee=e.
\]
Assume $q>p>e$. 
Then the sets of Anick $n$-chains for $n=1,2,3$ are easy to find:
\[
\begin{aligned}
V^{(1)}= {} & \{ [q|p],[q|e],[p|e],[e|q],[e|p],[e|e]\},
\\
V^{(2)}={} & \{ [q|p|e],[e|q|p],[q|e|p],[p|e|q],[q|e|e],[p|e|e],[e|e|q],[e|e|p], [e|q|e],\\
& [e|p|e],[q|e|q],[p|e|p],[e|e|e]\},
\\
V^{(3)}= {}& \{ [q|p|e|e],[e|q|p|e],[q|e|p|e],[p|e|q|e],[q|e|e|e],[p|e|e|e],\\ 
& [e|q|e|e], [e|p|e|e],[q|e|q|e],[p|e|p|e],[e|e|e|e],[e|e|q|p],[e|q|e|p],\\ 
& [e|p|e|q],[e|e|e|q],[e|e|e|p], [e|e|q|e], [e|e|p|e],[e|q|e|q],\\ 
& [e|p|e|p],[q|e|e|p],[p|e|e|q],[q|e|e|q],[p|e|e|p], [q|p|e|q], [q|p|e|p]\}.
\end{aligned}
\]
In order to compute $\mathrm H^3(W_1, M)$ for an arbitrary 
$W_1$-bimodule $M$ we need to know the Anick differentials on 
$V^{(2)}$ and $V^{(3)}$.

For example, consider a fragment of the graph constructed from the bar resolution of $\Lambda = W_1\oplus \Bbbk 1$
with the vertex $[q|p|e]$ with a matched edge $[p|q|e]\to [pq|e]$, see Fig.~\ref{Fig2}\,a. 
Note that $[p|q]$ is not an Anick chain thus should not be a critical cell. Indeed, the vertex $[p|q]$ belongs to another matched edge $[p|q]\to [pq]$ which also appears in the bar resolution graph, see Fig.~\ref{Fig2}\,b.
In a similar way, construct a fragment with the vertex $[e|q|p]$ on Fig.~\ref{Fig3}\,a:
all ending vertices of this fragment are either Anick chains or $[p|q]$ 
which is already matched. 
Note that the vertices $[e|p|q]$ and $[q|p|e]$ belong to 
matched edges.
As a final example, consider the fragment with $[e|q|p|e]$ (Fig.~\ref{Fig3}\,b): all ending vertices of this graph are either Anick chains or already matched ones.

In the sequel, we will often omit symbols $|$ in 
the elements of $V^{(n)}$.

\begin{figure}
\includegraphics{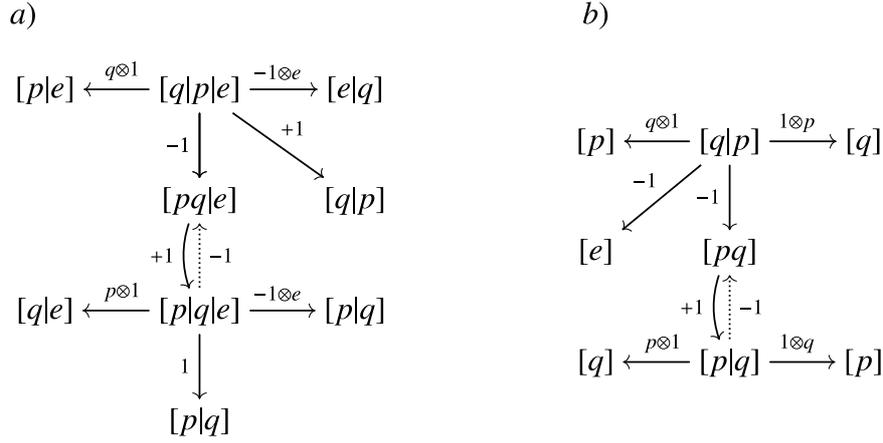}
\caption{Calculating the Anick differential of $[q|p|e]$ and $[q|p]$}\label{Fig2}
\end{figure}

\begin{figure}
\includegraphics{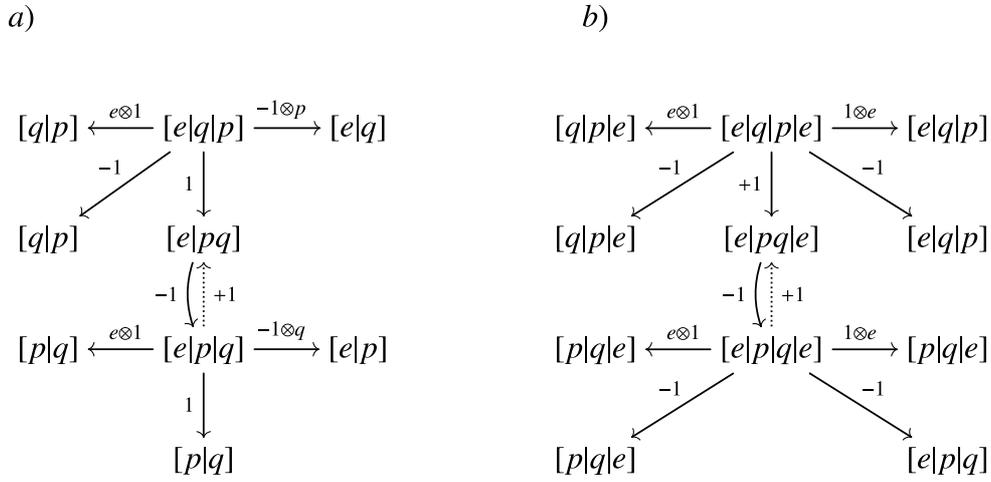}
\caption{Calculating the Anick differential of $[e|q|p]$ and $[e|q|p|e]$}\label{Fig3}
\end{figure}

In the same way, one may compute Anick differentials on the other 
chains from $V^{(2)}$ and $V^{(3)}$. As a resul, we get the following statements.

\begin{lemma}\label{lem:DiffV2}
The mapping $\delta _3: \mathrm A_3\to \mathrm A_2$ 
is defined by
\[
\begin{aligned}
  \delta_3[qpe]= {}&q[pe]-p[qe]-[ee]+[qp]-[qp]e, \\
  \delta_3[eqp]={}&e[qp]-[qp]+[ep]q+[ee]-[eq]p,\\
  \delta_3[qep]={}& q[ep]-[qe]p,\\
  \delta_3[peq]={}& p[eq]-[pe]q,\\
  \delta_3[qee]={}& q[ee]-[qe]e,\\
   \delta_3[pee]={}& p[ee]-[pe]e,\\
   \delta_3[eeq]={}& e[eq]-[ee]q,\\
    \delta_3[eep]={}& e[ep]-[ee]q,\\
     \delta_3[eqe]={}& e[qe]-[qe]+[eq]-[eq]e,\\
     \delta_3[epe] ={}& e[pe]-[pe]+[pe]-[ep]e,\\
     \delta_3[qeq]={}& q[eq]-[qe]q,\\
  \delta_3[pep]={}& p[ep]-[pe]p,\\
  \delta_3[eee] ={}& e[ee]-[ee]e.
\end{aligned}
\]
\end{lemma}

\begin{lemma}\label{lem:DiffV3}
The mapping $\delta _4: \mathrm A_4\to \mathrm A_3$ 
is defined by
\[
\begin{aligned}
  \delta_4[qpee]={}&q[pee]-p[qee]-[eee]+[qpe]e, \\
  \delta_4[qeep] ={}& q[eep]-[qep]+[qee]p,\\
  \delta_4[peeq]={}&  p[eeq]-[peq]+[pee]q,\\
  \delta_4[qeee] ={}& q[eee]-[qee]+[qee]e,\\
   \delta_4[peee]={}& p[eee]-[pee]+[pee]e,\\
   \delta_4[eeeq]={}& e[eeq]-[eeq]+[eee]q,\\
    \delta_4[eeep]={}& e[eep]-[eep]+[eee]p,\\
     \delta_4[eeqe]={}& e[eqe]-[eeq]+[eeq]e,\\
     \delta_4[eepe]={}& e[epe]-[eep]+[eep]e,\\
     \delta_4[qeeq]={}& q[eeq]-[qeq]+[qee]q,\\
  \delta_4[peep]={}& p[eep]-[pep]+[pee]p,\\
  \delta_4[eeqp]={}& e[eqp]-[eep]q-[eee]+[eeq]p,\\
 \delta_4[eqpe]={}& e[qpe]-[qpe]+[eee]-[eqp]+[eqp]e,\\ 
 \delta_4[qepe]={}& q[epe]-[qep]+[qep]e,\\
  \delta_4[peqe]={}& p[eqe]-[peq]+[peq]e,\\
   \delta_4[eqee]={}& e[qee]-[qee]+[eqe]e,\\
    \delta_4[epee]={}& e[pee]-[pee]+[epe]e,\\
     \delta_4[qeqe]={}& q[eqe]-[qeq]+[qeq]e,\\
     \delta_4[pepe]={}& p[epe]-[pep]+[pep]e,\\
     \delta_4[eeee]={}& e[eee]-[eee]+[eee]e,\\
     \delta_4[eqep]={}& e[qep]-[qep]+[eqe]p,\\
     \delta_4[epeq]={}& e[peq]-[peq]+[epe]q,\\
     \delta_4[epep]={}& e[pep]-[pep]+[epe]p,\\
     \delta_4[eqeq]={}& e[qeq]-[qeq]+[eqe]q,\\
     \delta_4[qpeq]={}& q[peq]-[eeq]-p[qeq]+[qpe]q,\\
      \delta_4[qpep]={}& q[pep]-[eep]-p[qep]+[qpe]p.
\end{aligned}
\]
\end{lemma}

\begin{theorem}\label{thm:WeylCohomology}
For an arbitrary $W_1$-bimodule $M$, the Hochchild cohomology group $\mathrm H^3(W_1,M)$ is trivial.
\end{theorem}

\begin{proof}
It is enough to find the respective cohomology group
of the complex 
$\Hom_{\Lambda{-}\Lambda} (\mathrm A_\bullet , M)$, 
where $\Lambda = W_1\oplus \Bbbk 1$, as above.

Note that an arbitrary bimodule $M$ over $W_1$ is a direct sum of four components:
\[
M = M_{1,1}\oplus M_{0,1}\oplus M_{1,0}\oplus M_{0,0},
\]
where the identity element $e\in W_1$ act on $M_{i,j}$
in such a way that 
$eu= iu$, $ue = ju$, for $u\in M_{i,j}$, 
$i,j\in \{0,1\}$. Hence, we may consider cohomologies 
with coefficients on the summands $M_{i,j}$ separately.

First, assume $M=M_{1,1}$, i.e., $eu=ue=u$ for all $u\in M$. 
Suppose $\varphi : \mathrm A_3 \to M$ is a cocycle, 
i.e., $\Delta^3(\varphi)=\varphi \delta_{4}=0$.
Apply $\varphi $ to all relations in Lemma~\ref{lem:DiffV3}: 
since zero emerges in all right-hand sides, we get the following 
relations on the values of $\varphi $ on the basis of $\mathrm A_3$
as of a free $\Lambda $-bimodule:
\begin{equation}\label{eq:CocycleRelations}
\begin{aligned}
\varphi[qpe]={}& -q\varphi[pee]+p\varphi[qee]+\varphi[eee], \\
\varphi[qep]={}& q\varphi[eep]+\varphi[qee]p,\\
\varphi[peq] ={}&p\varphi[eeq]+\varphi[pee]q,\\
  q\varphi[eee]={}&  p\varphi[eee] =  \varphi[eee]q 
            = \varphi[eee]p = \varphi[eqe] = \varphi[epe]=0,\\
    \varphi[qeq]&=q\varphi[eeq]+\varphi[qee]q,\\
\varphi[pep]&=p\varphi[eep]+\varphi[pee]p,\\ \varphi[eqp]&=\varphi[eep]q+\varphi[eee]-\varphi[eeq]p.
\end{aligned}
\end{equation}
As a corollary, 
\[
\varphi[eee] =e\varphi[eee]=q(p\varphi[eee])-p(q\varphi[eee])=0.
\]
Hence, $\varphi $ is completely determined by its values
\[
\varphi[eeq],\ \varphi[eep],\ \varphi[qee],\ \varphi[pee].
\]
Let us define $\psi\in \Hom_{\Lambda{-}\Lambda }(\mathrm A_2,M)$ 
in such a way that 
\[ 
\psi[eq]=\varphi[eeq],
\ \psi[ep]=\varphi[eep],
\ \psi[qe]=-\varphi[qee],
\ \psi[pe]=-\varphi[pee],
\]
and $ \psi[ee]=\psi[qp] = 0$.
Then $\Delta^2(\psi ) = \psi \delta_3$ is a coboundary, and 
\[
\begin{aligned}
(\psi \delta_3)[eeq]&=e\psi[eq]-\psi[ee]q=\varphi[eeq]+0=\varphi[eeq],\\
(\psi \delta_3)[eep]&=e\psi[ep]-\psi[ee]p=\varphi[eep]+0=\varphi[eep], \\
(\psi \delta_3)[qee]&=q\psi[ee]-\psi[qe]e=0+\varphi[qee]=\varphi[qee],\\
(\psi \delta_3)[pee]&=p\psi[ee]-\psi[pe]e=0+\varphi[pee]=\varphi[pee].
\end{aligned}
\]
Hence, $\Delta^2(\psi )= \varphi $, i.e., every 3-cocycle is a coboundary, so $\mathrm H^3(W_1,M)=0$ for every bimodule $M$ 
over~$W_1$.

Next, assume $M= M_{1,0}$, i.e., 
$eu=u$ and $ue=0$ for all $u\in M$. 
It follows from \ref{eq:CocycleRelations} that 
\begin{equation}\label{eq:CocycleRelations left}
\begin{gathered}
q\varphi[pee] -p\varphi[qee]-\varphi[eee]=0, 
\quad 
\varphi[qep]= q\varphi[eep],\quad 
\varphi[peq] = p\varphi[eeq],\\
\varphi[qee]= q[eee],\quad 
\varphi[pee]= p[eee],\quad
\varphi[eqe]= [eeq],\quad
\varphi[epe]= [eep],\\
\varphi[qeq]= q\varphi[eeq],\quad
\varphi[pep]= p\varphi[eep],\quad \varphi[eqp]=\varphi[eee].
\end{gathered}
\end{equation}
Therefore, $\varphi $ is completely determined by its values
$\varphi[eeq]$, $\varphi[eep]$, $\varphi[eee]$, $\varphi[qpe]$.
Let us define $\psi\in \Hom_{\Lambda }(\mathrm A_2,M)$ 
in such a way that 
\begin{gather*}
\psi[eq]=\varphi[eeq],
\ \psi[ep]=\varphi[eep],
\ \psi[qe]=\varphi[qee],\\
\ \psi[pe]=\varphi[pee],
\ \psi[ee]=\varphi[eee],
\ \psi[qp]=\varphi[qpe].
\end{gather*}
Then $\Delta^2(\psi ) = \psi \delta_3$ is a coboundary, and 
\[
\begin{aligned}
(\psi \delta_3)[eeq]&=e\psi[eq]=\psi[eq]=\varphi[eeq],\\
(\psi \delta_3)[eep]&=e\psi[ep]=\psi[ep]=\varphi[eep], \\
(\psi \delta_3)[eee]&=e\psi[ee]=\varphi[eee]=\varphi[eee],\\
(\psi \delta_3)[qpe]&=q\psi[pe]-p\psi[qe]-\psi[ee]+\psi[qp]\\
&=q\varphi[pee]-p\varphi[qee]-\varphi[eee]+\varphi[qpe]\\
&=0+\varphi[qpe]=\varphi[qpe].
\end{aligned}
\]
Hence, $\Delta^2(\psi )= \varphi $, i.e., every 3-cocycle is a coboundary, so $\mathrm H^3(W_1,M)=0$.

The cases of right-unital ($M_{0,1}$) and trivial 
($M_{0,0}$) modules are completely 
analogous.
\end{proof}

Since for every associative algebra $A$ and for every 
$A$-bimodule $M$ we have
$\mathrm H^{n+1}(A,M) = \mathrm H^{n}(A, \Hom(A,M))$, 
 all higher cohomologies (for $n\ge 3$) also vanish. 

\begin{corollary}
For every $n\ge 3$ we have $\mathrm H^n(W_1,M)=0$. 
\end{corollary}

{
The Hochschild cohomology is invariant under Morita equivalence of algebras,
and it is known that an algebra $A$ is Morita equivalent to the algebra of matrices $M_n(A)$ 
\cite{Keller},
\cite[Chapter~7]{Lam}, \cite[Chapter~1]{Loday} so
$\mathrm H^n(M_k(W_1),M)=\mathrm H^n(W_1,M)=0$.
}

As a corollary, we obtain the following description 
of conformal Hochschild cohomologies of 
the associative conformal algebra $\Cend_k$.

\begin{theorem}  
Let $M$ be a  conformal bimodule over 
$\Cend_k$, $k\ge 1$. 
Then
$\mathrm H^n( \Cend_k,M )=0$ for $n\ge 2$.
\end{theorem}

\begin{proof}
Proposition~\ref{prop:MainTool} immediately implies 
$\mathrm H^n( \Cend_k,M )=0$ for $n\ge 3$. 
For $n=2$, the result was obtained in \cite{Dolg2009}.
\end{proof}

\subsection*{Acknowledgments}
The work was supported by Russian Science Foundation, project 23-21-00504.

\end{document}